\documentclass[english, 11pt]{amsart}

\usepackage{amssymb}
\usepackage{amsfonts}
\usepackage{amscd}
\usepackage[T1]{fontenc}
\usepackage{color,mdwlist, enumitem}
\usepackage{hyperref}
\usepackage[margin=3cm]{geometry}

\def\11{{\mathbf 1}}

\def\cR{{\mathcal R}}

\def\R{{\mathbb R}}
\def\C{{\mathbb C}}

\newtheorem{theorem}{Theorem}
\newtheorem*{theorem*}{Theorem}

\newtheorem{corollary}{Corollary}

\newtheorem{fact}{Fact}

\newtheorem*{lthm}{Liouville's Theorem}

\theoremstyle{definition}

\newtheorem{def-prop}[subsubsection]{Proposition-Definition}
\newtheorem{def-theorem}[subsubsection]{Theorem-Definition}
\newtheorem{def-lem}[subsubsection]{Lemma-Definition}

\theoremstyle{remark}

\newtheorem{remark - ques}[subsubsection]{Remark/Quesion}

\theoremstyle{plain}

\begin{document}

\setcounter{tocdepth}{1} 

\title{A proof of Liouville's theorem via o-minimality}

\author{Pablo Cubides Kovacsics }
\address{Pablo Cubides Kovacsics, Universit\'e de Caen, Laboratoire de math\'ematiques Nicolas Oresme, CNRS UMR 6139,
14032 Caen cedex, France }
\email{pablo.cubides@unicaen.fr}
\thanks{${}^{\ast}$Supported by the ERC project TOSSIBERG (Grant Agreement 637027)}

\begin{abstract} In this short note we give a proof of Liouville's theorem (every bounded entire complex function is constant) following Peterzil and Starchenko's approach to complex analysis via o-minimality.  
\end{abstract}

\maketitle

In a series of papers \cite{peterzil-starchenko2001, peter-star2008}, Peterzil and Starchenko developed a setting in which they study analytic structures in algebraically closed fields of characteristic zero via definable sets in o-minimal expansions of real closed fields. For the case of $\R$ and $\C$, they provide several definable analogues of classic results in complex analysis. One of them is Liouville's theorem: every bounded entire complex function is constant. Their definable analogue states that every bounded entire complex function which is definable in an o-minimal expansion of $\R$, is constant. The latter result has strictly weaker assumptions than Liouville's classical statement since, by results of Peterzil and Starchenko, most entire complex functions are not definable in any o-minimal expansion of the reals (see later Theorem \ref{thm:defLiou2}). In this short note we show how to derive the full statement of Liouville's theorem still working in the o-minimal definable setting. 

\

We refer the reader to \cite{drie-1998} for definitions and basics of o-minimality. 

\section*{Complex analytic geometry via o-minimality}

Let $\cR$ be an o-minimal expansion of $(\R,+,\cdot)$. We identify $\C$ with the real plane $\R^2$, and by a $\cR$-definable subset of $\C^n$, we mean an $\cR$-definable subset of $\R^{2n}$ (with parameters). We denote by $|\cdot|\colon \C\to \R$ the usual absolute value, which is definable. The following result corresponds to the definable analogue of Liouville's theorem proved by Peterzil and Starchenko.  

\begin{theorem}[{\cite[Theorem 2.36 ]{peterzil-starchenko2001}}]\label{thm:defLiou} Let $f \colon  \C \to \C$ be a definable bounded entire function. Then $f$ is constant.
\end{theorem}

The fact that the previous theorem has strictly weaker assumptions than Liouville's classical theorem follows from the following result (and the fact that there are entire functions which are not rational functions).  

\begin{theorem}[{\cite[Theorem 1.2]{peterzil-starchenko2001}}]\label{thm:defLiou2}
Let $f \colon  \C \to \C$ be a definable entire function. Then $f$ is a rational function. 
\end{theorem}

It is worthy to notice that from Theorem \ref{thm:defLiou} one can still derive a proof of the fundamental theorem of algebra (every non-constant polynomial in $\C[X]$ has a root in $\C$).  

To recover Liouville's theorem we need to introduce the expansion $\cR_\text{an}$, which consists of $(\R,+,\cdot)$ together with functions $f\colon \R^n\to \R$ defined by
\[
f(x)=
\begin{cases}
g(x) & x\in [-1,1]^n\\
0 & \text{otherwise}, 
\end{cases}
\] 
where $g\colon [-1,1]^n\to \R$ is real-analytic. The following follows by classical results of Denef and van den Dries in \cite{denef-vdd-88}.  

\begin{theorem}[Denef-van den Dries]\label{thm:denef}
The structure $\cR_\text{an}$ is o-minimal and the definable sets in $\cR_{\text{an}}$ are precisely the globally subanalytic sets.
\end{theorem}

For $r$ a positive real number, denote by ${D}(r):=\{z\in \C: |z|\leqslant r\}$ the disc of radius $r$ in~$\C$ and by $C(r):=\{z\in \C: |z|=r\}$ the annulus of radius $r$. Set $D^\circ(r):=D(r)\setminus C(r)$. The following fact follows easily from Theorem \ref{thm:denef}. 

\begin{fact}\label{fact} Let $s$ be a positive real number and $f\colon \C\to D(s)$ be an entire function. Then, for every positive real number $r$, the restriction $f_{|D(r)}$ is definable in $\cR_\text{an}$. 
\end{fact}

\section*{Liouville's theorem}

The first step is to show a definable version of Schwarz Lemma. To that end, we will use the following particular case of the definable maximum principle from \cite{peterzil-starchenko2001}:  

\begin{theorem}[{\cite[Theorem 2.31]{peterzil-starchenko2001}}]\label{thm:max} Let $r$ be a positive real number and  $f \colon D(r)  \to \C$ be a definable continuous function which is $\C$-differentiable on $D^\circ(r)$. Then for all $z_0\in D(r)$
\[
|f(z_0)| \leq \max_{z\in C(r)} |f(z)|.
\]
\end{theorem}

\begin{theorem}[Definable Schwarz Lemma] Let $f\colon D(1)\to D(1)$ be a definable $\C$-differentiable function such that $f(0)=0$. Then $|f(x)|\leq |x|$. 
\end{theorem}

\begin{proof}
Consider the definable function 
\[
g(x)=
\begin{cases}
\frac{f(x)}{x} & x\neq 0\\
f'(0) & x=0.\\
\end{cases}
\]
Such function is definable and $\C$-differentiable. Hence, by Theorem \ref{thm:max}, we have that 
\[
\sup_{x\in D(1)} |g(x)| \leq \sup_{|x|=1} |g(x)| = \sup_{|x|=1} |f(x)| \leq 1.  
\]
Therefore $|g(x)|\leq 1$, which gives us that $|f(x)|\leq |x|$. 
\end{proof}

\begin{corollary}\label{cor} Let $r$ and $s$ be two positive real numbers and $f\colon D(r)\to D(s)$ be a definable $\C$-differentiable function such that $f(0)=0$. Then $|f(x)|\leq \frac{s}{r}|x|$. 
\end{corollary}

\begin{lthm} Let $f\colon  \C\to \C$ be an bounded entire function. Then $f$ is constant. 
\end{lthm}

\begin{proof} Consider the function $g(x)=f(x)-f(0)$. Clearly $g$ is also bounded and $g(0)=0$. Since $g$ is bounded, let $s$ be a positive real number such that $g\colon \C\to D(s)$. Now consider for every real number $r>0$ the restriction $g_r:=g_{|D(r)}$. By Fact \ref{fact}, for every $r>0$, the function~$g_r$ is definable in $\mathcal{R}_{an}$ and $\C$-differentiable by assumption. Hence, by Corollary \ref{cor}, we have that for all $x\in D(r)$, $|g_r(x)|\leq \frac{s}{r}|x|$. Letting $r$ go to infinity, we have that $g(x)=0$ for all $x\in \C$. Therefore, $f(x)=f(0)$ for all $x$. 
\end{proof}

\subsection*{Acknoledgements} I would like to thank the organizers of the JAVA (Jeunes en arithm\'etique et vari\'et\'es alg\'ebriques) meeting in 2016, as the ideas of this note emerged there.

\end{document}